\documentclass[journal]{IEEEtran}
\usepackage{graphics}
\usepackage{graphicx}
\usepackage{times}
\usepackage{amsmath}
\usepackage{amssymb}
\usepackage{flushend}
\usepackage{enumerate}
\usepackage[ruled,vlined]{algorithm2e}
\usepackage{subfigure}
\usepackage{comment}
\usepackage{amsthm}
\usepackage{cite}
\usepackage{color}
\usepackage{epstopdf}
\usepackage{epsfig}
\usepackage{bm}
\usepackage[thicklines]{cancel}

\title{Distributed Nash Equilibrium Seeking in Aggregative Games over Jointly Connected and Weight-Balanced Networks}

\author{Zhaocong~Liu,
	Jie~Huang,~\IEEEmembership{Life Fellow,~IEEE}
\thanks{This work was supported  by the Research Grants Council of the Hong Kong Special Administrative Region under grant No. 14201420.}
\thanks{The authors are with the Department of Mechanical and Automation
Engineering, The Chinese University of Hong Kong, Hong Kong.}
\thanks{Corresponding Author: Jie Huang ({\tt\small  jhuang@mae.cuhk.edu.hk}).}	

}

\newtheorem{theorem}{Theorem}

\newtheorem{lemma}{Lemma}

\newtheorem{remark}{Remark}

\newtheorem{assumption}{Assumption}

\newcommand{\col}{\hbox{col}}
\newcommand{\EQ}{\begin{eqnarray}}
	\newcommand{\EN}{\end{eqnarray}}
\newcommand{\EQQ}{\begin{eqnarray*}}
	\newcommand{\ENN}{\end{eqnarray*}}

\begin{document}
	
	\maketitle

	\begin{abstract}
		The problem  of the distributed Nash equilibrium seeking  for  aggregative games has been studied over strongly connected and weight-balanced static networks and every time strongly connected and weight-balanced switching networks. In this paper, we  further study the same problem over jointly connected and weight-balanced networks.
        The existing approaches critically rely on the connectedness of the network for constructing a Lyapunov function for their algorithms and theses approaches fail if the network is not connected. To overcome this difficulty, we propose an approach to show the exponential convergence of the output of the closed-loop system to the unknown Nash equilibrium (NE) point under a set of mild conditions.
	\end{abstract}


	\section{Introduction}
 The distributed  Nash equilibrium (NE) seeking problem is receiving an increasing attention from the control community  \cite{gadjov2018passivity}\cite{he2023Neurocomput}\cite{he2023tac}\cite{ye2017switching}\cite{ye2017distributed}, just to name a few.
The problem for  non-cooperative $N$-player game was studied in \cite{ye2017distributed} over static, undirected and connected networks and in \cite{ye2017switching} over every time strongly connected networks. Reference  \cite{gadjov2018passivity} further considered
the distributed  Nash equilibrium seeking problem over static, undirected and connected networks via a {\color{black}passivity-based} approach for games. In practice, the communication networks among the agents can be disconnected from time to time due to the changes of the environment or sensor/actuator failures. Thus, it is more interesting to consider switching communication networks which can be disconnected. In fact,
references \cite{he2023Neurocomput} and \cite{he2023tac} studied the distributed  Nash equilibrium seeking problem over jointly strongly connected switching networks  which can be disconnected at every time instant.

  Aggregative games are a {\color{black}subclass} of non-cooperative $N$-player games whose cost functions depend on a so-called {\color{black}aggregate function}.  Aggregative games have  been widely used to model the interaction between a group of self-interested players.
	The distributed NE seeking algorithms for aggregative games over static, undirected and connected communication networks were
proposed in, for example,  \cite{bianchi2021continuous}\cite{deng2021distributed}\cite{gadjov2020single}\cite{shakarami2022distributed}\cite{ye2016game}\cite{zhang2019distributed}.
The same problem was studied over static, strongly connected and weight-balanced networks in \cite{deng2018distributed}, over  strongly connected  and weight-balanced or weight-unbalanced directed networks in  \cite{zhu2020distributed}, and over  every time connected and undirected switching networks in \cite{liang2017distributed}.
Nevertheless, the approaches of the above mentioned papers do not apply to jointly connected switching networks because these approaches critically rely on the connectedness of the networks at every time.

In this paper, we will further consider the distributed NE seeking problem for  aggregative games over jointly connected and weight-balanced switching networks.
Compared with the existing literature, this paper offers the following features.
\begin{enumerate}[(i)]
		\item
	Our result applies to jointly connected and weight-balanced networks, which can be directed and disconnected at every time instant.  In contrast,  none of the existing methods mentioned above can handle disconnected networks.

\item Even for the special case where the graph is static, connected and weight-balanced as studied in \cite{deng2018distributed}\cite{deng2021distributed}, our main result also offers three significant improvements over the existing results  in \cite{deng2018distributed}\cite{deng2021distributed} as elaborated in Remark \ref{deng}.
  \end{enumerate}

To achieve the above advantages,  we propose a different  approach  from the existing ones in that we make use of the converse Lyapunov function theory to  construct the Lyapunov function candidate for our algorithm. For this purpose, we need to establish Lemma~\ref{lem: exponential stable},  which guarantees the exponential stability for a  time-varying ancillary  system.
	
	The rest of this paper is organized as follows.
	Section~\ref{sec:2} provides the preliminaries.
	Section~\ref{sec:3} presents the main result.
	Section~\ref{sec:6} closes the paper with some remarks.

	\indent\textbf{Notation}
	Let $\| \cdot \|: \mathbb{R}^n\to \mathbb{R}_{\geq 0}$ denote Euclidean norm and $\Vert \cdot \Vert:\mathbb{R}^{m\times n}\to \mathbb{R}_{\geq 0}$ denote the Euclidean-induced matrix norm.
	For column vectors $a_i$,  $\textrm{col} (a_1,\cdots,a_n)= [a_1^T,\cdots,a_n^T]^T$.
	For matrices $A_i$,  $\mathrm{blkdiag}(A_1,...,A_n)$ is the block diagonal matrix $
	\left[\begin{array}{ccc}
		A_1&&\\
		&\ddots&\\
		&&A_n
	\end{array}\right]$.
	$\otimes$ is the notation of the Kronecker product.
	\textcolor{black}{$\bm{1}_p$ is the $p$-dimensional column vector  with all $1$'s, $\bm{0}_{p\times q}$ is the $p\times q$-dimensional matrix with all $0$'s, and $I_p$ is the $p$-dimensional identity matrix.}

	\section{Preliminaries}\label{sec:2}

	\subsection{Game theory}	
	A non-cooperative game $\Gamma$ is defined by a triplet as follows:
	\begin{align}\label{static game obj}
		\Gamma \overset{\Delta}{=} (\mathcal{V},f_i,U_i)
	\end{align}
	where $\mathcal{V}=\{1,\cdots,N\}$ is the set of $N$ players.
	For each player $i \in  \mathcal{V}$, the action of player $i$ is denoted by  $x_i \in U_i \subset \mathbb{R}^n$ where $U_i$ is called the action space of player $i$.
	Let  $U =U_1\times U_2 \times \cdots \times U_N \subset \mathbb{R}^{Nn}$ and  ${\bm x}=\mathrm{col}(x_1,x_2,\cdots,x_N) \in \mathbb{R}^{Nn}$,  which are called the  action space and the strategy vector of the game, respectively. Then, $f_i: U \to \mathbb{R}$ is the cost function for player $i$.
	Define ${\color{black}\bm{x}_{-i}}\triangleq (x_1,x_2, \cdots, x_{i-1},x_{i+1}, \cdots, x_N)$, and $U_{-i} \overset{\Delta}{=} U_1\times \cdots \times U_{i-1} \times U_{i+1} \cdots \times U_N$.
	Then, the goal of each player $i$ is, for all ${\color{black}\bm{x}_{-i} \in U_{-i}}$,  to minimize its cost function $f_i (x_i, \bm{x}_{-i})$ over $x_i \in U_i$, that is,
	\begin{align}\label{minimize fi}
		\text{ minimize }  f_i (x_i, \bm{x}_{-i}) \hspace{3mm} \text{ subject to }  x_i \in U_i
	\end{align}
Let $\nabla_if_i(x_i,\bm{x}_{-i}) =  [\frac{\partial f_i(x_i,\bm{x}_{-i})}{\partial x_i}]^T \in\mathbb{R}^n$ be the partial derivative of $f_i$ with respect to $x_i$. Then, we call
	\begin{align}\label{psudo-gradient}
		F({\boldsymbol x})= {\col} \left( \nabla_1 f_1 (x_1,\bm{x}_{-1}), \cdots, \nabla_N f_N (x_N,\bm{x}_{-N})  \right)
	\end{align}
	the pseudo-gradient of the game $\Gamma$.

In this paper, {\color{black}we focus on aggregative games whose cost functions depend on a so-called aggregate function which is defined as follows:}
	\begin{align}\label{def sigma}
		\textcolor{black}{\check{\sigma}(\boldsymbol{x})}  \overset{\Delta}{=} \frac{1}{N}\sum_{i=1}^{N} \phi_i(x_i)
	\end{align}
where $\phi_i(\cdot): \mathbb{R}^{n}\mapsto \mathbb{R}^n$ is a continuously differentiable vector-valued function which represents
the local contribution to the aggregated quantity.
An aggregative game is a game whose cost functions $f_i(x_i,\bm{x}_{-i}) = \bar{f}_i(x_i,\check{\sigma} (\boldsymbol{x}))$ for some functions $\bar{f}_i$.
A strategy vector ${\bm x}^*=(x_i^*,\bm{x}_{-i}^*) \in U$ is  a Nash equilibrium of the aggregative game  if it is such that \begin{align}\label{aggregative NE def}
		\textcolor{black}{\bar{f}_i(x_i^*,\check{\sigma}(x_i^*,\bm{x}_{-i}^*)) \leq \bar{f}_i(x_i,\check{\sigma}(x_i,\bm{x}_{-i}^*))},~~ \forall i \in \mathcal{V}, \forall x_i \in U_i.
	\end{align}

To study the  aggregative game,
define the following functions:
	\begin{align}\label{parial gradient Ji Def}
		J_i(x_i,s_i) \overset{\Delta}{=} \nabla_y \bar{f}_i(y,s_i)|_{y=x_i} + \frac{1}{N} \nabla\phi_i(x_i)\nabla_y \bar{f}_i(x_i,y)|_{y=s_i}
	\end{align}
where $\nabla\phi_i(x_i)\in\mathbb{R}^{n\times n}$ is evaluated at $x_i$ and equals the transpose of the Jacobian matrix of vector function $\phi_i(x_i)$, i.e., $\nabla\phi_i(x_i) = (\frac{\partial \phi_i(y)}{\partial y}|_{y=x_i})^T$.
	
	Let $\bm{s} = \mathrm{col}(s_1,s_2,\cdots,s_N)\in \mathbb{R}^{Nn}$, and $\phi(\bm{x}) = \mathrm{col}(\phi_1(x_1),\phi_2(x_2),\cdots,\phi_N(x_N))\in\mathbb{R}^{Nn}$.
	Then the following operator
	\begin{align}\label{mathbf F def}
		\mathbf{F}(\boldsymbol{x},\boldsymbol{s}) = \textrm{col}(J_1(x_1,s_1),\cdots,J_N(x_N,s_N))
	\end{align}
is called the \emph{extended pseudo-gradient} operator.  Then the fact that \textcolor{black}{$f_i(x_i,\bm{x}_{-i}) = \bar{f}_i(x_i,\check{\sigma} (\boldsymbol{x}))$},
and equations \eqref{psudo-gradient}, \eqref{parial gradient Ji Def} and \eqref{mathbf F def} imply  that $\mathbf{F}(\bm{x},\bm{s}) = F(\bm{x})$ if $\textcolor{black}{\bm{s} = 1_N \otimes \check{\sigma}(\bm{x})} = (\frac{1_N 1_N^T}{N} \otimes I_n)\phi(\bm{x})$.

{\color{black} Three standard assumptions are as follows \cite{bianchi2021continuous}\cite{deng2018distributed}\cite{gadjov2020single}. }
\begin{assumption}\label{assump 1}
~~
	\begin{enumerate}[1)]
		
		\item {\color{black} For all $i \in \mathcal{V}$, } $U_i$ is nonempty, closed and convex.
		
		\item The cost function $f_i(x_i,\bm{x}_{-i})$ is convex and
		continuously differentiable in $x_i$ for every fixed $\bm{x}_{-i}\in  U_{-i}$.
		\item  The pseudo-gradient $F$ is strongly monotone on $U$, that is, for some $\mu>0$, \begin{equation}
			(\boldsymbol{x}-\boldsymbol{x'})^T(F(\boldsymbol{x})-F(\boldsymbol{x'})) \geq \mu \|\boldsymbol{x}-\boldsymbol{x'}\|^2, \forall \boldsymbol{x},\boldsymbol{x'} \in U \nonumber
		\end{equation}
	\end{enumerate}
\end{assumption}

\begin{assumption}\label{assump 2}
	The pseudo-gradient $F$ is Lipschitz continuous, i.e., for some $\theta>0$,
	\begin{equation}
		\|F(\bm{x}) - F(\bm{x'})\| \leq \theta\|\bm{x}-\bm{x'}\|, \forall \bm{x},\bm{x'} \in U  \nonumber
	\end{equation}
\end{assumption}

\begin{assumption}\label{assump 3}
	    For all $\bm{x}\in U$,
	    \begin{enumerate}[1)]
	    	\item  The extended pseudo-gradient $\mathbf{F}$ is Lipschitz continuous in its second variable, that is, for some $\hat{\theta}>0$, $\|\mathbf{F}(\bm{x},\bm{s}) - \mathbf{F}(\bm{x},\bm{s'})\|\leq \hat{\theta}\|\bm{s}-\bm{s'}\|, \forall \bm{s},\bm{s'}\in\mathbb{R}^{Nn}$.
	    	
	    	\item The Jacobian of $\phi(\bm{x})$ is globally bounded, i.e.,  $\|\frac{\partial \phi(\bm{x})}{\partial \bm{x}}\|\leq l$ for some $l>0$.
	    \end{enumerate}
\end{assumption}

    \begin{remark}\label{Rem: Fx*=0}

    	If Parts $1)$ and $2)$ of Assumption~\ref{assump 1} hold, then by \cite[Proposition~1.4.2]{facchinei2003finite},  an NE $\bm{x^*}$ exists which is such that the following variational inequality $VI(U,F)$ holds:
    	\begin{align}\label{VI def}
    		(\bm{x}-\bm{x^*})^T F(\bm{x^*}) \geq 0, \forall \bm{x}\in U
    	\end{align}		
    	Moreover, by \cite[Theorem~2.3.3 (b)]{facchinei2003finite}, Part $3)$ of Assumption~\ref{assump 1} guarantees a unique NE $\bm{x^*}$ exists.   	
    	In what follows,  we consider the global case, that is,  $U = \mathbb{R}^{Nn}$. For this case,  \eqref{VI def} implies  $F(\bm{x^*})=\bm{0}_{(Nn)\times 1}$.

    Assumption~\ref{assump 2} is weaker than the smoothness requirement in \cite[Assumption~1]{liang2017distributed} and \cite[Assumption~1]{zhu2020distributed}.
    Part $2)$ of Assumption~\ref{assump 3} includes average aggregate function $\phi_i(x_i) = x_i$ and linear weighted aggregate function $\phi_i(x_i) = A_ix_i$ with matrices $A_i\in\mathbb{R}^{n\times n}$ as special cases.
    \end{remark}

	\subsection{Graph theory}
	
The information exchange of all players of the game described in (\ref{static game obj}) can be described by a time-varying  graph\footnote{See Appendix for a summary of graph.} $\mathcal{G}_{\sigma (t)} =(\mathcal{V},\mathcal{E}_{\sigma (t)})$ with $\mathcal {V}=\{1,\dots,N\}$, $\sigma(t)$  a piece-wise constant switching signal, and $\mathcal{E}_{\sigma (t)} \subseteq\mathcal{V}\times\mathcal{V}$ for all $t \geq 0$. For any $t \geq 0$, $\mathcal{E}_{\sigma (t)}$ contains an edge $(j, i)$ if and only if the player $i$ is able to use  the information of player $j$ at time $t$.	
\textcolor{black}{We define the neighbor set of agent $i$ at time $t$ as $\mathcal{N}_i(t)=\{j\in\mathcal{V}|(j,i)\in\mathcal{E}_{\sigma(t)}\}$.}
A graph $\mathcal{G}_{\sigma(t)}$ is called weight-balanced at time $t$ if $\sum_{j\in\mathcal{V}} a_{ij}(t) = \sum_{j\in\mathcal{V}} a_{ji}(t)$ holds for all $i\in\mathcal{V}$. For any $t\geq 0, s>0$, let $\textcolor{black}{\mathcal{G}_{\sigma([t,t+s))}} = \cup_{t_i\in[t,t+s)} \mathcal{G}_{\sigma(t_i)}$.
We call $\mathcal{G}_{\sigma([t,t+s))}$ the \emph{union graph} of $\mathcal{G}_{\sigma(t)}$ over the time interval $[t,t+s)$.

	We have the following assumption regarding the communication of the players.
	\begin{assumption}\label{assump 4}
~~~
\begin{enumerate}[1)]
	    	\item  {\color{black}There exists a  positive number $T$ such that the  graph $\mathcal{G}_{\sigma ([t,t+T))}$ is connected for all $t \geq 0$.}
	    	
	    	\item  The graph $\mathcal{G}_{\sigma(t)}$ is weight-balanced for any $t\geq 0$.
	    \end{enumerate}
	\end{assumption}

	\begin{remark}
		A time-varying graph satisfying Part 1) of Assumption~\ref{assump 4} is called jointly connected.
		Under Assumption~\ref{assump 4}, the graph can be directed and disconnected at every time instant.
		Therefore, none of the existing approaches in \cite{bianchi2021continuous}\cite{de2019continuous}\cite{deng2018distributed}\cite{deng2021distributed}\cite{gadjov2020single}\cite{liang2017distributed}\cite{shakarami2022distributed}\cite{ye2016game}\cite{zhang2019distributed}\cite{zhu2020distributed} applies to this case. {\color{black} It is also interesting to note that Assumption~\ref{assump 4} implies the  graph $\mathcal{G}_{\sigma ([t,t+T))}$ is strongly connected for all $t \geq 0$\cite[Lemma~17]{wu2005algebraic}. }
\end{remark}

	\section{Main Result}\label{sec:3}	
	
Let us first propose our distributed NE seeking algorithm for player $i$ as follows:
	\begin{subequations}\label{ctrl law 4 single integrator}
		\begin{align}
			\dot{x}_i &= -\delta J_i(x_i, s_i) \label{ctrl law 4 single integrator 1} \\
			\dot{s}_i &= -\alpha(s_i-\phi_i(x_i)) - \beta\sum_{j\in \mathcal{N}_i(t)} (s_i-s_j) - \nu_i \label{ctrl law 4 single integrator 2}\\
			\dot{\nu}_i &= \alpha\beta\sum_{j\in \mathcal{N}_i(t)} (s_i-s_j)  \label{ctrl law 4 single integrator 3}
		\end{align}
	\end{subequations}
	where $s_i \in \mathbb{R}^n, \nu_i \in \mathbb{R}^n$ are two variables to ensure exact estimation of the aggregate value $\sigma(\bm{x})$, and $\delta, \alpha, \beta$ are three adjustable parameters to be specified later.
	
	\begin{remark}
	The algorithm \eqref{ctrl law 4 single integrator} is	motivated  by \cite{deng2018distributed}, which is in turn  inspired by the dynamic average consensus algorithm proposed in \cite{kia2015dynamic}. {\color{black} However, we have somehow simplified \cite[Equation~(11d)]{deng2018distributed} by removing the term  $\nabla\phi_i(x_i)$ to obtain the current form of the equation \eqref{ctrl law 4 single integrator 2}. In contrast to the original form in \cite[Equation~(4a)]{kia2015dynamic}, we also remove $\frac{d}{dt} \phi_i(x_i) = \frac{\partial \phi_i(x_i)}{\partial x_i}\dot{x}_i$ in \eqref{ctrl law 4 single integrator 2}. This change simplifies the analysis below and reduces the calculation burden.   }
\end{remark}

	Let $\bm{x} = \mathrm{col}(x_1,\cdots,x_N)\in \mathbb{R}^{Nn}, \bm{s} = \mathrm{col}(s_1,\cdots,s_N)\in \mathbb{R}^{Nn}, \bm{\nu} = \mathrm{col}(\nu_1,\cdots,\nu_N)\in \mathbb{R}^{Nn}$.
	Then, the concatenated form of \eqref{ctrl law 4 single integrator} is as follows:
	\begin{subequations}\label{compact 4 single integra}
		\begin{align}
			\dot{\bm{x}} &= -\delta \mathbf{F}(\bm{x},\bm{s})   \label{compact 4 single integra 1} \\
			\dot{\bm{s}} &= -\alpha(\bm{s}-\phi(\bm{x})) - \beta(\mathcal{L}_{\sigma(t)}\otimes I_n)\bm{s} - \bm{\nu}   \label{compact 4 single integra 2} \\
			\dot{\bm{\nu}} &= \alpha\beta(\mathcal{L}_{\sigma(t)}\otimes I_n)\bm{s} \label{compact 4 single integra 3}
		\end{align}
	\end{subequations}

	Now we concentrate on the subsystem composed of \eqref{compact 4 single integra 2} and \eqref{compact 4 single integra 3}.
	First, we define two projection matrices as follows:
	\begin{subequations}\label{project opera}
		\begin{align}
			P_n &= \frac{\bm{1}_N \bm{1}_N^T}{N} \otimes I_n  \label{project opera 1} \\
			P_n^\perp &= I_{Nn} -  \frac{\bm{1}_N \bm{1}_N^T}{N} \otimes I_n  \label{project opera 2}
		\end{align}
	\end{subequations}
	In fact, $P_n$ in \eqref{project opera 1} denotes projection onto consensus subspace of dimension $n$, and $P_n^\perp$ in \eqref{project opera 2} represents projection onto disagreement subspace of dimension $n$.
	
	Consider the following coordinate transformation:
	\begin{subequations}\label{error Def 4 bms and bmnu}
		\begin{align}
			\bar{\bm{s}} &= \bm{s}-1_N \otimes \textcolor{black}{\check{\sigma}(\bm{x})} = \bm{s}- P_n\phi(\bm{x}) \label{error Def 4 bms and bmnu 1} \\
			\bar{\bm{\nu}} &= \bm{\nu}-\alpha(I_{Nn}\!-\! \frac{\bm{1}_N \bm{1}_N^T}{N} \otimes I_n) \phi (\bm{x}) = \bm{\nu}\!-\! \alpha P_n^\perp\phi(\bm{x}) \label{error Def 4 bms and bmnu 2}
		\end{align}
	\end{subequations}
where $\bar{\bm{s}}=\textrm{col}(\bar{s}_1,\bar{s}_2,\cdots,\bar{s}_N), \bar{\bm{\nu}} = \textrm{col}(\bar{\nu}_1,\bar{\nu}_2,\cdots,\bar{\nu}_N)$.
	Then, \eqref{compact 4 single integra 2}-\eqref{compact 4 single integra 3} is equivalent to the following:
	\begin{subequations}\label{error system 4 bms and bmnu}
		\begin{align}
			\dot{\bar{\bm s}} &= -\alpha\bm{s} + \alpha \phi (\bm{x}) - \beta(\mathcal{L}_{\sigma(t)}\otimes I_n)\bm{s} - \bm{\nu}  - P_n\frac{\partial \phi(\bm{x})}{\partial \bm{x}}\dot{\bm{x}} \nonumber \\
			&= -\alpha\bar{\bm{s}} \!-\! \beta(\mathcal{L}_{\sigma(t)}\otimes I_n)\bar{\bm{s}} \!-\! \bar{\bm{\nu}} \!+\! \delta P_n\frac{\partial \phi(\bm{x})}{\partial \bm{x}}\mathbf{F}(\bm{x},\bar{\bm{s}}+ P_n  \phi (\bm{x}))  \label{error system 4 bms and bmnu 1} \\
			\dot{\bar{\bm{\nu}}} &= \alpha\beta(\mathcal{L}_{\sigma(t)}\otimes I_n)( \bar{\bm{s}} + {\color{black} \bm{1}_N \otimes \check{\sigma}(\bm{x}) } )  - \alpha P_n^\perp\frac{\partial \phi(\bm{x})}{\partial \bm{x}}\dot{\bm{x}} \notag \\
			&=  \alpha\beta(\mathcal{L}_{\sigma(t)}\otimes I_n)\bar{\bm{s}} + \delta\alpha P_n^\perp\frac{\partial \phi(\bm{x})}{\partial \bm{x}}\mathbf{F}(\bm{x},\bar{\bm{s}}+ P_n  \phi (\bm{x}))  \label{error system 4 bms and bmnu 2}
		\end{align}
	\end{subequations}
	The Jacobian linearization  of \eqref{error system 4 bms and bmnu} at the origin is as follows:
	\begin{subequations}\label{linear part without xdot}
		\begin{align}
			\dot{\bar{\bm s}} &= -\alpha\bar{\bm{s}} - \beta(\mathcal{L}_{\sigma(t)}\otimes I_n)\bar{\bm{s}} - \bar{\bm{\nu}} \label{linear part without xdot 1} \\
			\dot{\bar{\bm{\nu}}} &= \alpha\beta(\mathcal{L}_{\sigma(t)}\otimes I_n)\bar{\bm{s}}  \label{linear part without xdot 2}
		\end{align}
	\end{subequations}

	Let  $r = \frac{\bm{1}_N}{\sqrt{N}}\in \mathbb{R}^N$.  Then, there exists $R \in \mathbb{R}^{N \times (N-1)}$ such that $R^TR = I_{N-1}$ and $R^Tr = \bm{0}_{(N-1)\times 1}$. That is,   the matrix $\mathcal{Q} = \begin{bmatrix}
		r & R
	\end{bmatrix}$ is an orthogonal matrix.
	Let $r_\otimes = r\otimes I_n, R_\otimes = R\otimes I_n, \mathcal{Q}_\otimes  = \mathcal{Q}\otimes I_n$.
	We further define the following coordinate transformation
\begin{subequations}\label{transformation}
\begin{align}
\bm{y} &= \mathcal{Q}_\otimes^T\bar{\bm{s}} = \begin{bmatrix}
		\bm{y}_1 \\
 \bm{y}_2
	\end{bmatrix} \\
\bm{z} &= \mathcal{Q}_\otimes^T\bar{\bm{\nu}}= \begin{bmatrix}
		\bm{z}_1 \\
 \bm{z}_2
	\end{bmatrix}
\end{align}
	\end{subequations}
 with $\bm{y}_1,\bm{z}_1\in\mathbb{R}^n, \bm{y}_2,\bm{z}_2\in\mathbb{R}^{Nn-n}$.
    {\color{black} Suppose Part 2) of Assumption~\ref{assump 4} is satisfied. Then,  for all $t\geq 0$,  $\bm{1}_N^T\mathcal{L}_{\sigma(t)} = \bm{0}_{1\times N}$, which implies
 $r^T_\otimes  (\mathcal{L}_{\sigma(t)}\otimes I_n) = \bm{0}_{n\times (nN)}$.} Thus, system \eqref{linear part without xdot} is equivalent to the following
	\begin{subequations}\label{linear part with y1y2z1z2}
		\begin{align}
			\dot{\bm{y}}_1 &= -\alpha\bm{y}_1 - \beta r^T_\otimes  (\mathcal{L}_{\sigma(t)}\otimes I_n)\bar{\bm s}  - \bm{z}_1 = -\alpha\bm{y}_1  - \bm{z}_1   \label{linear part with y1y2z1z2 1} \\
			\dot{\bm{y}}_2 &= -\alpha\bm{y}_2 - \beta ((R^T\mathcal{L}_{\sigma(t)}R)\otimes I_n) \bm{y}_2 - \bm{z}_2   \label{linear part with y1y2z1z2 2} \\
			\dot{\bm{z}}_1 &= \alpha\beta r^T_\otimes  (\mathcal{L}_{\sigma(t)}\otimes I_n) \bar{\bm s}  = \bm{0}_{n\times 1} \label{linear part with y1y2z1z2 3} \\
			\dot{\bm{z}}_2 &= \alpha\beta ((R^T\mathcal{L}_{\sigma(t)}R)\otimes I_n) \bm{y}_2 \label{linear part with y1y2z1z2 4}
		\end{align}
	\end{subequations}	

	To study the stability property of \eqref{linear part with y1y2z1z2}, consider   the following ancillary system:
	\begin{align}\label{linear switching sys with zeta}
		\dot{\bm{\zeta}} &= A(t)\bm{\zeta}
	\end{align}
	where $\bm{\zeta}=\mathrm{col}(\zeta_1,\zeta_2,\zeta_3) $ with $\zeta_1 \in \mathbb{R}^{n}$, $\zeta_2,  \zeta_3 \in \mathbb{R}^{Nn-n}$, and
	\begin{align}\label{linear switching sys 4 At}
		A(t) \!=\! \begin{bmatrix}
			-\alpha I_n \!&\! \bm{0} \!&\! \bm{0}\\
			\bm{0} \!&\! -\alpha I_{Nn-n}\!-\!\beta( R^T\mathcal{L}_{\sigma(t)}R)\!\otimes \!I_n \!&\! -I_{Nn-n} \\
			\bm{0} \!&\! \alpha\beta ((R^T\mathcal{L}_{\sigma(t)}R)\otimes I_n) \!&\! \bm{0}
		\end{bmatrix}
	\end{align}

	We first establish the following lemma.
	\begin{lemma}\label{lem: exponential stable}
	 {\color{black}	Under  Part $1)$ of Assumption~\ref{assump 4}, } the origin of the linear switched system~\eqref{linear switching sys with zeta} is exponentially stable.
	\end{lemma}
	\begin{proof}
		Let  $\hat{\zeta} = \textrm{col}(\hat{\zeta}_1, \hat{\zeta}_2,\cdots, \hat{\zeta}_N)$ with $\hat{\zeta}_i \in\mathbb{R}^{n}$, $i = 1, \cdots, N$. Then, we first consider the following subsystem:
\begin{align}\label{prof of lem exponential stable 1}
			\dot{\hat{\zeta}}=-\beta(\mathcal{L}_{\sigma(t)}\otimes I_n)\hat{\zeta}
		\end{align}
Under Part $1)$ of Assumption~\ref{assump 4},   by \cite[Corollary~2.1]{cai2022cooperative}, which in turn follows from \cite[Theorem~1]{moreau2004stability},
all $\hat{\zeta}_i$ converge to a common vector $\xi\in\mathbb{R}^n$ exponentially as $t\to +\infty$, that is, $\lim\limits_{t\to +\infty} (\hat{\zeta}(t) - \bm{1}_N\otimes \xi) = \bm{0}_{(Nn)\times 1}$ exponentially.
	
Next, define coordinate transformation $\tilde{\zeta} = \textrm{col}(\tilde{\zeta}_1, \tilde{\zeta}_2) = \begin{bmatrix}
			r_\otimes^T \\ R_\otimes^T
		\end{bmatrix}\hat{\zeta}$.
	{\color{black}Then, using the property that $\mathcal{L}_{\sigma(t)}\bm{1}_N = \bm{0}_{N\times 1}$ for $t\geq 0$,  \eqref{prof of lem exponential stable 1} is transformed to the following form:
		\begin{subequations}\label{prof of lem exponential stable 2}
					\begin{align}
						\dot{\tilde{\zeta}}_1 &= -\beta((r^T\mathcal{L}_{\sigma(t)}R)\otimes I_n)\tilde{\zeta}_2  \label{prof of lem exponential stable 2(1)} \\
						\dot{\tilde{\zeta}}_2 &= -\beta((R^T\mathcal{L}_{\sigma(t)}R)\otimes I_n)\tilde{\zeta}_2 \label{prof of lem exponential stable 2(2)}
					\end{align}
		\end{subequations}

We now show the origin of \eqref{prof of lem exponential stable 2(2)} is exponentially stable. }
Note that,
\begin{align*}
			\tilde{\zeta}_2(t) &=  R_\otimes^T  \hat{\zeta}(t) \\
&= R_\otimes^T (\hat{\zeta}(t) -  \bm{1}_N\otimes \xi ) + R_\otimes^T (\bm{1}_N\otimes \xi) \\
& = R_\otimes^T (\hat{\zeta}(t) -  \bm{1}_N\otimes \xi )
			\end{align*}
\textcolor{black}{where the third equality follows from $R_\otimes^T (\bm{1}_N\otimes \xi) = (R^T\bm{1}_N)\otimes (I_n\xi) = \bm{0}_{(Nn-n)\times 1}$. }	
		
Thus, for any initial condition $\tilde{\zeta}_2(0)\in\mathbb{R}^{Nn-n}$,
		\begin{align*}
			\lim\limits_{t\to +\infty} \tilde{\zeta}_2(t) &=  R_\otimes^T\lim\limits_{t\to +\infty} \hat{\zeta}(t) \\
&= R_\otimes^T\lim\limits_{t\to +\infty} (\hat{\zeta}(t) -  \bm{1}_N\otimes \xi )  \\
& =  \bm{0}_{(Nn-n) \times 1}
		\end{align*}
exponentially,  which means $\tilde{\zeta}_2$ tends to the origin exponentially as $t$ goes to infinity.
		
		We now  show that the origin of the linear switched system~\eqref{linear switching sys with zeta} is exponentially stable.
		For this purpose, let $w=\alpha\zeta_2+\zeta_3$. Then
		\eqref{linear switching sys with zeta} is transformed to the following form:
\begin{subequations}\label{prof of lem exponential stable 3}
			\begin{align}
				\dot{{\zeta}}_1 &= - \alpha {\zeta}_1 \label{prof of lem exponential stable 3(1)} \\
				\dot{\zeta}_2 &= -\beta((R^T\mathcal{L}_{\sigma(t)}R)\otimes I_n)\zeta_2 - w \label{prof of lem exponential stable 3(2)} \\
			\dot{w} &= - \alpha w \label{prof of lem exponential stable 3(3)}
\end{align}
\end{subequations}
Thus, both ${\zeta}_1$ and
 $w$ vanish exponentially. Since \eqref{prof of lem exponential stable 3(2)}
 can be viewed as an exponentially stable linear system  perturbed by an exponentially vanishing input $w$,
by \cite[Corollary~2.4]{cai2022cooperative} or \cite[Lemma~1]{liu2018leader} , we have $\lim\limits_{t\to +\infty} \zeta_2(t) = \bm{0}_{(Nn-n)\times 1}$ exponentially.
The proof is thus complete.		
	\end{proof}

	\begin{remark}\label{Rem: differ Kia}
		It is  interesting to compare  Lemma~\ref{lem: exponential stable} with \cite[Lemma~4.4]{kia2015dynamic} where it was  showed that $\bar{s}_i$ in \eqref{linear part without xdot 1} and $\bar{\nu}_i$ in \eqref{linear part without xdot 2} achieve consensus exponentially, respectively.
First,  in \cite{kia2015dynamic}, $\phi_i (x_i) = x_i$ with $x_i$ a scalar while  our Lemma~\ref{lem: exponential stable} works for a more general $\phi_i (x_i)$ with $x_i$ a vector. Thus, Lemma~\ref{lem: exponential stable} here has somehow extended \cite[Lemma~4.4]{kia2015dynamic}. Moreover,
 we further showed that a reduced system governing  only $\bm{y}_1, \bm{y}_2$ and $\bm{z}_2$ is exponentially stable. This result is crucial for constructing the Lyapunov function~\eqref{prof thm1 4 Lyapunov v1 plus v2} for the system (\ref{prof thm1 4 modified closed-loop}) in the proof of Theorem~\ref{thm: main 1} later. 
\end{remark}

	Now we are ready to establish our main result which makes use of  the converse Lyapunov theorem based on Lemma~\ref{lem: exponential stable}.

	\begin{theorem}\label{thm: main 1}
		Under Assumptions~\ref{assump 1} - \ref{assump 4}, there exists a constant $\delta^* > 0$ such that, for any $0<\delta<\delta^*, \alpha,\beta>0$, any $x_i(0)\in \mathbb{R}^n, s_i(0)\in \mathbb{R}^n$, and $\sum_{i=1}^{N} \nu_i(0) = \bm{0}_{n\times 1}$, the solution of the  system \eqref{compact 4 single integra} is bounded over $t\geq 0$ and satisfies:
		\begin{subequations}\label{main thm 1 converge 4 single integra}
			\begin{align}
				\lim_{t\to+\infty} \bm{x}(t) &= \bm{x^*} \label{main thm 1 converge 4 single integra 1} \\
				\lim_{t\to+\infty} \bm{s}(t) &= P_n\phi(\bm{x^*}) = (\frac{\bm{1}_N \bm{1}_N^T}{N} \otimes I_n) \phi(\bm{x^*}) \label{main thm 1 converge 4 single integra 2} \\
				\lim_{t\to+\infty} \boldsymbol{\nu}(t) &= \alpha P_n^\perp\phi(\bm{x^*}) = \alpha(I_{Nn}- \frac{\bm{1}_N \bm{1}_N^T}{N} \otimes I_n)\phi(\bm{x^*}) \label{main thm 1 converge 4 single integra 3}
			\end{align}
		\end{subequations}
all exponentially.
	\end{theorem}
	\begin{proof}			
		Let $\textrm{col}(\bm{x},\bar{\bm{s}},\bar{\bm{\nu}})$ be governed by \eqref{compact 4 single integra 1}, \eqref{error system 4 bms and bmnu 1},  and \eqref{error system 4 bms and bmnu 2}, respectively. Let
\begin{align}
\begin{bmatrix}
				  \bm{z}_1 \\	\bm{y} \\ \bm{z}_2
				\end{bmatrix} &=
\begin{bmatrix}
	                 \bm{0}_{n\times (Nn)} &  r_\otimes^T  \\
					\mathcal{Q}_\otimes^T & \bm{0}_{Nn}   \\
                     \bm{0}_{(Nn-n) \times (Nn)}    & R_\otimes^T
				\end{bmatrix}
\begin{bmatrix}
					\bar{\bm{s}} \\ \bar{\bm{\nu}}
				\end{bmatrix}
\end{align}
Then, \textcolor{black}{under Part 2) of Assumption~\ref{assump 4}}, using \eqref{linear switching sys 4 At} and the fact that $R_\otimes^T  P_n=\bm{0}_{(Nn-n) \times (Nn)}$  and $r_\otimes^TP^\perp_n=\bm{0}_{n \times (Nn)}$, \eqref{compact 4 single integra 1}, \eqref{error system 4 bms and bmnu 1},  and \eqref{error system 4 bms and bmnu 2} are transformed to the following:
\begin{subequations}\label{prof thm1 4 modified closed-loop}
			\begin{align}
				\dot{\bm{x}} &= -\delta\mathbf{F}(\bm{x},\mathcal{Q}_\otimes\bm{y}+P_n\phi(\bm{x})) \label{prof thm1 4 modified closed-loop x}  \\
				\dot{\bm{z}}_1 &=  \bm{0}_{n\times 1}   \label{prof thm1 4 modified closed-loop z1} \\
				\frac{d}{dt}\!\begin{bmatrix}
					\bm{y}_1\\ \bm{y}_2 \\ \bm{z}_2
				\end{bmatrix} &= A(t)\!\begin{bmatrix}
					\bm{y}_1\\ \bm{y}_2 \\ \bm{z}_2
				\end{bmatrix}   \nonumber  \\
&\hspace{4mm} \!+\! \begin{bmatrix}
					\delta r_\otimes^T\frac{\partial \phi(\bm{x})}{\partial \bm{x}}\mathbf{F}(\bm{x}, \mathcal{Q}_\otimes\bm{y} \!+\! P_n\phi(\bm{x})) \!-\! \bm{z}_1 \\
					\bm{0}_{(Nn-n)\times 1}   \\
					\delta\alpha R_\otimes^T\frac{\partial \phi(\bm{x})}{\partial \bm{x}}\mathbf{F}(\bm{x}, \mathcal{Q}_\otimes\bm{y} \!+\! P_n\phi(\bm{x}))
				\end{bmatrix}    \label{prof thm1 4 modified closed-loop y1y2z2}
			\end{align}
		\end{subequations}
		
Since $\sum_{i=1}^{N} \nu_i(0) = \bm{0}_{n\times 1}$, we have
		\begin{align}\label{main thm1 4 z10=0}
			\bm{z}_1(0) &= r_\otimes^T\bar{\bm{\nu}}(0) = r_\otimes^T(\bm{\nu}(0)\!-\!\alpha P_n^\perp\phi(\bm{x}(0))) = r_\otimes^T\bm{\nu}(0) = \bm{0}_{n\times 1}
		\end{align}
Thus, by \eqref{prof thm1 4 modified closed-loop z1},  we have $\bm{z}_1 (t) =\bm{0}_{n\times 1}$ for all $t \geq 0$.
As a result, the linear part  of the subsystem~\eqref{prof thm1 4 modified closed-loop y1y2z2} is given by (\ref{linear switching sys with zeta}), whose equilibrium at the origin is exponentially stable by Lemma \ref{lem: exponential stable}.
Let $\Phi(\tau,t)$ be the state transition matrix of (\ref{linear switching sys with zeta}). Then there exist some positive constants $\gamma$ and $\lambda$ such that
		\begin{align}
			\|\Phi(\tau,t)\| \leq \gamma e^{-\lambda(\tau-t)}, \forall \tau\geq t
		\end{align}
		Define $P(t) = \int_{t}^{\infty} \Phi^T(\tau,t)Q\Phi(\tau,t)d\tau$ with $Q$ being some constant positive definite matrix.
		Then, similar to \cite[Lemma~3.1]{liu2017adaptive}, one can verify that $P(t)$ is continuous for all $t\geq 0$, and it is positive definite and decrescent in the sense that there exist constants $c_1, c_2>0$ such that
		\begin{align}\label{main thm1 4 time-varying pd and decrescent V}
			c_1\|v\|^2\leq v^T P(t)v\leq c_2\|v\|^2
		\end{align}
		It implies that $\|P(t)\|\leq p$ with some positive constant $p$ for all $t\geq 0$.		
		Note that $\mathcal{L}_{\sigma(t)}$ is a piece-wise constant matrix with the range of $\sigma(t)$ being a finite set $\mathcal{P}=\{1,2,\cdots,n_0\}$.
		Thus $A(t)$ in \eqref{linear switching sys 4 At} is bounded over $[0,+\infty)$ and continuous on each time interval $[t_{j},t_{j+1})$ for $j=0,1,\cdots$.
		Therefore, for $t\in[t_j,t_{j+1}), j=0,1,2,\cdots$, the following holds \cite[Theorem~4.12]{khalil002}:
		\begin{align}\label{time-varing Sylvester eqn}
			-\dot{P}(t) = A(t)^TP(t) + P(t)A(t) + Q
		\end{align}

		Let $\bar{\bm{x}} = \bm{x}-\bm{x^*}, \bm{\xi}=\mathrm{col}(\bm{y},\bm{z}_2)$, $V_1(\bar{\bm{x}}) =  \frac{1}{2}\|\bm{x}-\bm{x^*}\|^2$, and
 $V_2(\bm{\xi},t) = \bm{\xi}^T	P(t) \bm{\xi}$. Then, we define a time-varying Lyapunov function candidate for system \eqref{prof thm1 4 modified closed-loop} as follows:
		\begin{align}\label{prof thm1 4 Lyapunov v1 plus v2}
			V(\bar{\bm{x}},\bm{\xi},t) &= V_1(\bar{\bm{x}}) + V_2(\bm{\xi}, t)
			\end{align}	
		The derivative of $V_1$ along \eqref{prof thm1 4 modified closed-loop x} satisfies
		\begin{align}\label{prof thm1 4 derivative of V1}
			\dot{V}_1 &= (\bm{x}-\bm{x^*})^T(-\delta\mathbf{F}(\bm{x},\mathcal{Q}_\otimes\bm{y}+P_n\phi(\bm{x}))) \notag \\
			&\overset{(a)}{=}  -\delta(\bm{x}-\bm{x^*})^T(\mathbf{F}(\bm{x},\mathcal{Q}_\otimes\bm{y}+P_n\phi(\bm{x}))-\mathbf{F}(\bm{x},P_n\phi(\bm{x}))) \nonumber \\ & \hspace{5mm}- \delta(\bm{x}-\bm{x^*})^T(F(\bm{x})-F(\bm{x^*})) \nonumber \\
			&\overset{(b)}{\leq}  \delta\hat{\theta}\|\bar{\bm{x}}\|\|\mathcal{Q}_\otimes\bm{y}\| -  \delta\mu\|\bar{\bm{x}}\|^2 \nonumber \\
			&\overset{(c)}{\leq} \delta\hat{\theta}\|\bar{\bm{x}}\|\|\bm{\xi}\| - \delta\mu\|\bar{\bm{x}}\|^2
		\end{align}	
		where equality $(a)$ follows from $\mathbf{F}(\bm{x}, P_n\phi(\bm{x})) = \mathbf{F}(\bm{x},(\frac{\bm{1}_N \bm{1}_N^T}{N} \otimes I_n)\phi(\bm{x})) = F(\bm{x})$ by \eqref{mathbf F def} and $F(\bm{x^*})=\bm{0}_{(Nn)\times 1}$ by Remark~\ref{Rem: Fx*=0}, inequality $(b)$ follows from Part $3)$ of  Assumption~\ref{assump 1} and part $1)$ of Assumption~\ref{assump 3}, and inequality $(c)$ follows from $\|\mathcal{Q}_\otimes\bm{y}\|=\|\bm{y}\|\leq\|\begin{bmatrix}
			\bm{y}\\\bm{z}_2
		\end{bmatrix}\|$.

			On the other hand, for any $t\in[t_j,t_{j+1}), j=0,1,2,\cdots$, taking derivative of $V_2$ with respect to \eqref{prof thm1 4 modified closed-loop y1y2z2} yields
			\begin{align}\label{prof thm1 4 derivative of V2}
				\dot{V}_2
				=& \begin{bmatrix}
					\bm{y}^T & \bm{z}_2^T
				\end{bmatrix}(A(t)^TP(t)+P(t)A(t)+\dot{P}(t))\begin{bmatrix}
					\bm{y} \\ \bm{z}_2
				\end{bmatrix} \nonumber \\
				&  + 2\begin{bmatrix}
					\bm{y}^T & \bm{z}_2^T
				\end{bmatrix}P(t)\begin{bmatrix}
					\delta r_\otimes^T\frac{\partial \phi(\bm{x})}{\partial \bm{x}}\mathbf{F}(\bm{x}, \mathcal{Q}_\otimes\bm{y}+ P_n\phi(\bm{x}))\\
					\bm{0}_{(Nn-n)\times 1} \\ \delta\alpha R_\otimes^T\frac{\partial \phi(\bm{x})}{\partial \bm{x}}\mathbf{F}(\bm{x}, \mathcal{Q}_\otimes\bm{y}+ P_n\phi(\bm{x}))
				\end{bmatrix} \notag \\
				\overset{(a)}{\leq}&   -\bm{\xi}^TQ\bm{\xi} + \notag \\
				& \hspace{-1.5mm} 2\delta\sqrt{\alpha^2\!+\!1}\|\bm{\xi}\|\|P(t)\|\|\mathcal{Q}_\otimes^T\frac{\partial \phi(\bm{x})}{\partial \bm{x}}\|\|\mathbf{F}(\bm{x},\mathcal{Q}_\otimes\bm{y}\!+\!P_n\phi(\bm{x}))\| \notag \\
				\overset{(b)}{\leq}& -\lambda_{min}(Q)\|\bm{\xi}\|^2 \!+\! 2pl\delta\sqrt{\alpha^2+1}\|\bm{\xi}\|(\hat{\theta}\|\bm{y}\|\!\!+\!\!\theta\|\bm{\bar{x}}\|)
			\end{align}
			where inequality $(a)$ follows from \eqref{time-varing Sylvester eqn}, and inequality $(b)$ follows from following facts: $\|\mathcal{Q}_\otimes^T\| = 1, \|\mathbf{F}(\bm{x},\mathcal{Q}_\otimes\bm{y}+P_n\phi(\bm{x}))\| \leq \|\mathbf{F}(\bm{x},\mathcal{Q}_\otimes\bm{y}+P_n\phi(\bm{x}))-\mathbf{F}(\bm{x},P_n\phi(\bm{x}))\| + \|\mathbf{F}(\bm{x},P_n\phi(\bm{x}))-\mathbf{F}(\bm{x^*},P_n\phi(\bm{x^*}))\| \leq \hat{\theta}\|\bm{y}\|+\theta\|\bm{\bar{x}}\|$ by Remark~\ref{Rem: Fx*=0}, Assumption~\ref{assump 2} and Part $1)$ of Assumption~\ref{assump 3}, $\|\frac{\partial \phi(\bm{x})}{\partial \bm{x}}\|\leq l$ for all $\bm{x}\in\mathbb{R}^{Nn}$ by Part $2)$ of Assumption~\ref{assump 3}, and $\|P(t)\|\leq p$ for all $t\geq 0$.

			Define $M = 2pl\sqrt{\alpha^2+1} >0$.
			Then \eqref{prof thm1 4 derivative of V2} can be further simplified as follows:
			\begin{align}\label{prof thm1 4 simplfied derivative of V2}
				\dot{V}_2 &\leq -\lambda_{min}(Q)\|\bm{\xi}\|^2 + \delta M\hat{\theta}\|\bm{\xi}\|^2 + \delta M\theta\|\bm{\xi}\|\|\bm{\bar{x}}\| \notag \\
				&\leq  -(\lambda_{min}(Q)-\delta M\hat{\theta})\|\bm{\xi}\|^2 + \delta M\theta\|\bm{\xi}\|\|\bm{\bar{x}}\|
			\end{align}

			Substituting \eqref{prof thm1 4 derivative of V1} and \eqref{prof thm1 4 simplfied derivative of V2} into $\dot{V} = \dot{V}_1 + \dot{V}_2$ gives
			\begin{align}\label{prof thm1 4 total derivative}
				&\dot{V} =  \dot{V}_1 + \dot{V}_2 \notag \\
				\leq& - \delta \begin{bmatrix}
					\|\bar{\bm{x}}\| \\ \|\bm{\xi}\|
				\end{bmatrix}^T\underbrace{\begin{bmatrix}
						\mu & -\frac{(\hat{\theta}+M\theta)}{2} \\
						-\frac{(\hat{\theta}+M\theta)}{2} & \frac{\lambda_{min}(Q)}{\delta} -  M\hat{\theta}
				\end{bmatrix}}_{B(\delta)}\begin{bmatrix}
					\|\bar{\bm{x}}\| \\ \|\bm{\xi}\|
				\end{bmatrix}
			\end{align}
			Since $\mu > 0$,  the matrix $B(\delta)$ in \eqref{prof thm1 4 total derivative} is positive definite if $\mu(\frac{\lambda_{min}(Q)}{\delta}- M\hat{\theta})-\frac{(\hat{\theta}+M\theta)^2}{4} >0.$
			Therefore, we can select the positive constant $\delta^*$ as
			\begin{align}\label{exact expre 4 delta* of single integra}
				\delta^* = \frac{4\mu\lambda_{min}(Q)}{(\hat{\theta}+M\theta)^2+4\mu M\hat{\theta}}
			\end{align}
			such that, for any $0<\delta<\delta^*$,  $\dot{V} \leq -\delta\lambda_{min}(B(\delta))\|\begin{bmatrix}
				\|\bar{\bm{x}}\| \\ \|\bm{\xi}\|
			\end{bmatrix} \|^2$.
			Since both $V(\bar{\bm{x}},\bm{\xi},t)$  and $- \dot{V}(\bar{\bm{x}},\bm{\xi},t)$ are positive definite quadratic functions in $\mbox{col~} (\bar{\bm{x}},\bm{\xi})$, we have
$\lim\limits_{t\to +\infty} \bar{\bm{x}}(t) = \bm{0}_{(Nn)\times 1}, \lim\limits_{t\to +\infty} \bm{y}(t) = \bm{0}_{(Nn)\times 1}, \lim\limits_{t\to +\infty} \bm{z}_2(t) = \bm{0}_{(Nn-n)\times 1}$ all exponentially.
			Therefore, we have
			\begin{align}\label{main thm1 4 respective final converg x}
				\lim\limits_{t\to +\infty} \bm{x}(t) = \lim\limits_{t\to +\infty} (\bar{\bm{x}}(t) + \bm{x^*}) = \bm{x^*}
			\end{align}
			exponentially, which means the strategy vector $\bm{x}$ tends to the NE exponentially as $t$ goes to infinity.
			Since $\bm{z}_1(t)$ is identically zero for all $ t \geq 0$, using \eqref{error Def 4 bms and bmnu}, \eqref{transformation} and  \eqref{main thm1 4 respective final converg x},  we can further obtain 	
			\begin{subequations}\label{main thm1 4 respective final converg s and nu}
				\begin{align}
					&\lim\limits_{t\to +\infty} \bm{s}(t) = \lim\limits_{t\to +\infty} (\mathcal{Q}_\otimes\bm{y}(t) + P_n\phi(\bm{x}(t))) = P_n\phi(\bm{x^*}) \\
					&\lim\limits_{t\to +\infty} \bm{\nu}(t) = \lim\limits_{t\to +\infty} (\mathcal{Q}_\otimes\bm{z}(t) + \alpha P_n^\perp\phi(\bm{x}(t))) = \alpha P_n^\perp\phi(\bm{x^*})
				\end{align}
			\end{subequations}
both exponentially.
		\end{proof}

		\begin{remark}
			Reference \cite{nian2021distributed} studied general multi-cluster game problem over switching networks.
			For comparison, let  $m_i=1$ for all $i\in\mathcal{V}$ in \cite{nian2021distributed}, that is,  there is only one player in each cluster.
			Then   the first line of \cite[Equation~(21)]{nian2021distributed} can be put as follows:
			\begin{align}\label{nian2021 same as yemaojiao switching }
				\dot{\bm{\hat{q}}} = -\delta\textrm{diag}\{c_{ij}\}((\mathcal{L}_{\sigma(t)}\otimes I_N+A_0)\hat{\bm{q}} + r_0)
			\end{align}
			where $\bm{\hat{q}}=\textrm{col}(\hat{q}_{11},\cdots,\hat{q}_{1N},\cdots,\hat{q}_{NN})$ with $\hat{q}_{ij}$ being player $i$'s estimate of action of player $j$, $c_{ij}>0$ are the gains, and $r_0=-\textrm{col}(a_{11}^{\sigma(t)}q_1,a_{12}^{\sigma(t)}q_2,\cdots,a_{1N}^{\sigma(t)}q_N,\cdots,a_{NN}^{\sigma(t)}q_N)$.
			One immediately realizes that it is the same as the second equation of \cite[Equation~(9)]{ye2017switching}. Thus,  by \cite[Remark~3.4]{he2023Neurocomput}, the approach in \cite{nian2021distributed}  only applies to every time strongly connected networks.		
		\end{remark}

		\begin{remark}\label{deng}
		Our result also applies to the special case where the networks are static, connected and weight-balanced as studied in, for example,  \cite{deng2018distributed} and \cite{deng2021distributed}. Even for this special case,  our result offers a few advantages as follows:
 \begin{enumerate}
 \item
 Our algorithm relaxes the restrictive initial condition $\nu_i(0)=\bm{0}_{n\times 1} (i\in\mathcal{V})$ in \cite[Equation~(11e),(13e)]{deng2018distributed} and \cite[Equation~(7e),(24e)]{deng2021distributed} to merely $\sum_{i=1}^N\nu_i(0) = \bm{0}_{n\times 1}$, which significantly enlarges the feasible set of initial conditions.
  \item The validity of \cite[Lemma~1]{deng2018distributed} or \cite[Lemma~1]{deng2021distributed} relies on a crucial assumption that  $\phi_i(x_i)$ in \eqref{ctrl law 4 single integrator 2} are either  constants  or exponentially converge to some constants, which may not be verifiable since $\bm{x}$ dynamics and $\bm{s}-\bm{\nu}$ dynamics are coupled in \eqref{compact 4 single integra}. In contrast, we do not need such an assumption.
			
\item We gave an explicit upper bound $\delta^*$  in \eqref{exact expre 4 delta* of single integra} for  the design parameter $\delta$ while \cite{deng2018distributed}\cite{deng2021distributed} only assumed the existence
of the upper bound $\delta^*$.
\end{enumerate}
\end{remark}

		\section{Conclusion}\label{sec:6}
		In this paper, we have studied the problem  of the distributed Nash equilibrium seeking  for  aggregative games over jointly connected and weight-balanced switching networks. The existing approaches critically rely on the connectedness of the graph for constructing a Lyapunov function for their algorithms and theses approaches fail if the network is not connected. To overcome this difficulty,  we have first established the exponential stability for a time-varying ancillary system. Then, by the converse Lyapunov theorem, we obtain a time-varying quadratic Lyapunov function for the ancillary system,  which in turn leads to the construction of a suitable Lyapunov function for the closed-loop system, thus leading to the solution of the problem.

		{\appendix
A time-varying graph is denoted by $\mathcal{G}(t) = (\mathcal{V},\mathcal{E}(t))$, where
$\mathcal{V} = \{1,\cdots,N\}$ is the node set, and $\mathcal{E}(t)\subseteq \mathcal{V}\times \mathcal{V}$ is the
edge set.
If there is an edge from node $j$ to node $i$, then $(j,i)\in \mathcal{E}(t)$, and we say $j$ is a neighbor of $i$ at time $t$.

A subset of $\mathcal{E}(t)$ of the form $\{(i_1,i_2),\cdots,(i_{k-1},i_k)\}$ is said to be a directed path from node $i_1$ to node $i_k$, and in this case, we say node $i_1$ can reach node $i_k$ at time $t$. The graph $\mathcal{G}(t)$ is said to be static if $\mathcal{G}(t)=\mathcal{G}(0)$ for all $t\geq 0$.
A static graph is denoted by $\mathcal{G}(\mathcal{V},\mathcal{E})$.
A static graph is said to be \emph{connected} if there is a node that can reach every other node, and is said to be \emph{strongly connected} if there is a directed path between any two nodes.
An edge $(i,j)$ is called an undirected edge if $(i,j)\in \mathcal{E} \Leftrightarrow (j,i)\in \mathcal{E}$.
$\mathcal{G}$ is called an undirected graph if every edge in $\mathcal{E}$ is undirected.
The adjacency matrix of a graph $\mathcal{G}$ is a non-negative matrix $\mathcal{A} = [a_{ij}]\in \mathbb{R}^{N\times N}$ where, for $i,j=1,\cdots,N, a_{ij}=1$ if there is an edge from node $j$ to node $i$
and $a_{ij}=0$ if otherwise. 
Since, for $i\in\mathcal{V}$, there is no such edge as $(i,i)$, we have $a_{ii}=0$.
For $i=1,\cdots,N$, let $d_i^{in} = \sum_{j=1}^N a_{ji}$ and $d_i^{out} = \sum_{j=1}^N a_{ij}$, which are called the in-degree and out-degree of node $i$, respectively. Let $D = \mathrm{diag}(d_1^{out},\cdots,d_N^{out})$, which is called the degree matrix of $\mathcal{G}$. The matrix $\mathcal{L} = D - \mathcal{A}$ is called Laplacian of $\mathcal{G}$ corresponding to $\mathcal{A}$.

A time function $\sigma: [0,+\infty) \mapsto \mathcal{P}=\{1,\cdots,n_0\}$ with
$n_0$ being some positive integer is said to be a piece-wise
constant switching signal if there exists a sequence $\{t_j:j=0,1,\cdots\}$ satisfying $t_0 = 0$ and $t_{j+1}-t_j\geq \tau$ for some positive constant $\tau$ such that, for all $t\in [t_j,t_{j+1}), \sigma(t) = p$ for some $p\in \mathcal{P}$.
$\mathcal{P}$ is called the switching index set, $t_j$ is called the switching instant, and $\tau$ is called the dwell time.

Given a set of $r$ graphs $\{\mathcal{G}_i = (\mathcal{V}, \mathcal{E}_i), i=1,\cdots,r\}$, the graph $\mathcal{G} = (\mathcal{V},\mathcal{E})$ where $\mathcal{E} = \cup_{i=1}^r \mathcal{E}_i$ is called the union of graphs $\mathcal{G}_i$, denoted by $\mathcal{G} = \cup_{i=1}^r \mathcal{G}_i$.

Given a piece-wise constant switching signal $\sigma(t)$ and a set of $n_0$ static graphs $\mathcal{G}_i = (\mathcal{V}, \mathcal{E}_i), i =1,\cdots,n_0$, one can define a time-varying graph $\mathcal{G}_{\sigma(t)} = (\mathcal{V}, \mathcal{E}_{\sigma(t)})$, which is called a \emph{switching graph}.
}

	\end{document}